\documentclass[12pt,a4paper]{amsart}
\usepackage{amssymb,enumerate,mathtools,cite}
\usepackage{graphicx,epstopdf}
\usepackage{caption,subcaption,float}

\DeclareMathOperator{\RE}{Re}
\DeclareMathOperator{\IM}{Im}
 \def  \E{e^{i \theta}}
 \def \a{\alpha}

  \def \vU {\varUpsilon}

\numberwithin{equation}{section}
\newtheorem{theorem}{Theorem}[section]
\newtheorem{lemma}[theorem]{Lemma}
\newtheorem{corollary}[theorem]{Corollary}
\theoremstyle{remark}
\newtheorem{remark}{Remark}[section]
\newcommand{\D}{\mathbb{D}}
\newcommand{\C}{\mathbb{C}}

\allowdisplaybreaks

\begin{document}
\title[Sharp Bohr Radius Constants...]{Sharp Bohr Radius Constants For Certain  Analytic Functions }
\author[S. Anand]{Swati Anand}

\address{Department of Mathematics, University of Delhi,
Delhi--110 007, India}
\email{swati\_anand01@yahoo.com}
\author [N.K. Jain]{Naveen Kumar Jain}
\address {Department of Mathematics, Aryabhatta College, Delhi-110021,India}
\email{naveenjain05@gmail.com}
\author [S. Kumar]{Sushil Kumar}
\address {Bharati Vidyapeeth's college of Engineering, Delhi-110063, India}
\email{sushilkumar16n@gmail.com}

\begin{abstract}The Bohr radius for a class $\mathcal{G}$  consisting of analytic functions $f(z)=\sum_{n=0}^{\infty}a_nz^n$ in unit disc $\mathbb{D}=\{z\in\mathbb{C}:|z|<1\}$ is the largest $r^*$ such that every function $f$ in the class $\mathcal{G}$ satisfies the inequality
\begin{equation*}
d\left(\sum_{n=0}^{\infty}|a_nz^n|, |f(0)|\right) = \sum_{n=1}^{\infty}|a_nz^n|\leq d(f(0), \partial f(\mathbb{D}))
\end{equation*}
for   all $|z|=r \leq r^*$, where $d$ is the Euclidean distance.
In this paper, our aim is  to determine the  Bohr radius for the classes of  analytic functions $f$ satisfying  differential subordination relations $zf'(z)/f(z) \prec h(z)$ and  $f(z)+\beta z f'(z)+\gamma z^2 f''(z)\prec h(z)$, where $h$ is the Janowski function.  Analogous  results are  obtained for the classes of $\alpha$-convex functions and typically real functions, respectively. All obtained results are sharp.

%The main purpose of this investigation is to determine the  Bohr radius for the class of  Janowski starlike functions and the class of analytic functions $f$ satisfying second order differential subordinaton $f(z)+\beta z f'(z)+\gamma z^2 f''(z)\prec h(z)$, where $h$ is a Janowski  function. The Bohr radius is also obtained for the class of $\alpha$-convex functions and typically real functions. All obtained radii are sharp.

\end{abstract}



\maketitle
\section{Introduction}\label{sec1}
In recent years, the Bohr radius problems attracted the attention of several researchers in various direction in geometric function theory. The Bohr inequality has emerged as an active area of research after Dixon \cite{D1}  used it to disprove a conjecture in Banach algebra.
Let $\mathbb{D}:=\{z\in\mathbb{C}:|z|<1\}$  be the open  unit disc in $\mathbb{C}$ and $\mathcal{H}(\mathbb{D}, \Omega)$ denote the class of analytic functions mapping unit disc into  a domain $\Omega$. Let $\mathcal{A}$ denote the class of analytic functions in $\mathbb{D}$ normalized by $f(0)=0=f'(0)-1.$  Let $\mathcal{S}$ denote the subclass of $\mathcal{A}$ consisting of univalent functions. For  two analytic functions $f$ and $g$  in $\mathbb{D}$,  the function $f$ is said to be subordinate to $g$, written $f(z)\prec g(z)$, if there is an analytic  map $w:\mathbb{D}\rightarrow\mathbb{D}$ with $w(0)=0$ satisfying $f(z)=g(w(z))$. In particular, if the function $g$ is univalent in $\mathbb{D}$, then $f$ is subordinate to $g$ is equivalent to $f(0)=g(0)$ and $f(\mathbb{D})\subset g(\mathbb{D})$.
Let $\phi$ be the analytic function with positive real part
in $\mathbb{D}$ that map the unit disc $\mathbb{D}$ onto regions starlike with respect to $1$, symmetric with respect
to the real axis and normalized by the conditions $\phi(0) = 1$ and $\phi'(0) > 0$. For such functions, Ma and Minda \cite{ma} introduced the following classes:
\[\mathcal{ST}(\phi):=\left\{f\in\mathcal{A};\frac{zf'(z)}{f(z)}\prec\phi(z)\right\}\quad\text{and}\quad
\mathcal{CV}(\phi):=\left\{f\in\mathcal{A};1+\frac{zf''(z)}{f'(z)}\prec\phi(z)\right\}.\]
On taking $\phi(z)=(1+Az)/(1+Bz)(-1\leq B<A\leq1)$, the class $\mathcal{ST}(\phi)$ reduces to  the familiar class consisting of  Janowski starlike functions \cite{jan1}, denoted by $\mathcal{ST}[A,B]$. The special case $A=1-2\alpha$ and $B=-1$, that is, $\phi(z)=(1+(1-2\alpha)z)/(1-z)$, $0\leq\alpha< 1$ yield the classes $\mathcal{ST}(\alpha)$ and $\mathcal{CV}(\alpha)$ of  starlike and convex functions of order $\alpha$, respectively. In particular, $\alpha = 0$, that is, $A = 1$ and $B = −1$ leads to the usual classes $\mathcal{ST}$ and $\mathcal{CV}$ of starlike and convex functions, respectively.
For $A=1$ and $B=(1-M)/M$, $M>1/2$, we obtain the class $\mathcal{ST}(M)$ introduced by Janowski\cite{jan2}. In 1939, Robertson \cite{rob}  introduced a very well known class $\mathcal{ST}^{(\beta)}:=\mathcal{ST}[\beta, -\beta] $, $0<\beta\leq 1$. Also, $ \mathcal{ST}_{(\beta)}:=\mathcal{ST}[\beta,0]$ leads to a class which was introduced by  {MacGregor \cite{mac}}.

In 1914,  Bohr \cite{B1}  discovered that if a power series of an analytic function converges in the unit disc and its sum has a modulus less than one, then  the sum of the absolute values of its terms is again less than one in the disc  $|z|\leq 1/6$.
%As Bohr noted in his paper,
Wiener, Riesz and Schur  independently proved that the Bohr's result holds in the disc $|z|\leq 1/3$ and the number $1/3$ is best possible. For the class of functions $f \in \mathcal{H}(\mathbb{D}, \mathbb{D})$, the number $1/3$ is commonly called the Bohr radius, while the inequality $\sum_{n=0}^\infty |a_nz^n|\leq1$ is known as the Bohr inequality. Later on,  various proof of Bohr's inequality were given in \cite{P2, P3, P4}.
Using the Euclidean distance $d$, the  Bohr inequality for a function  $f(z)=\sum_{n=0}^\infty a_nz^n$ is written  as
\[d\left(\sum_{n=0}^{\infty}|a_nz^n|, |a_0|\right)=\sum_{n=1}^{\infty}|a_nz^n|\leq 1-|f(0)|= d(f(0), \partial \mathbb{D}), \]
where  $\partial \mathbb{D}$ is the boundary of the disc $\mathbb{D}$.
For any domain $\Omega$  and all functions $f \in \mathcal{H}(\mathbb{D},\Omega)$, the Bohr radius is the  largest radius $r^* > 0$ such that
\begin{equation*}\label{eqn6}
d\left(\sum_{n=0}^{\infty}|a_nz^n|, |f(0)|\right) = \sum_{n=1}^{\infty}|a_nz^n|\leq d(f(0), \partial \Omega),
\end{equation*}
for $|z|=r \leq r^*$. Let the class $\mathcal{G}$  consisting of analytic functions $f(z)=\sum_{n=0}^{\infty}a_nz^n$ which map disc $\mathbb{D}$ into a domain $\Omega$. Then Bohr radius for the class $\mathcal{G}$ satisfies the inequality
\[d\left(\sum_{n=0}^{\infty}|a_nz^n|, |f(0)|\right) = \sum_{n=1}^{\infty}|a_nz^n|\leq d(f(0), \partial f(\mathbb{D}))\] for all  $|z|=r \leq r^*$. In this case, the class $\mathcal{G}$ is said to satisfy a Bohr phenomenon.

Ali et al. \cite{Alex}
obtained the Bohr radius for wedge domain  $W_{\alpha} = \{w :| \arg w| < \pi\alpha/2\}, 1\leq \alpha \leq 2$ and also determined the upper  and lower bounds on the Bohr radius for odd analytic functions. Further, several different improved versions of the classical Bohr inequality were given in \cite{Ponnu5}. Alkhaleefah \emph{et al.} \cite{Ponnu1} established the classical Bohr inequality for the class of quasisubordinate functions, while the classical Bohr inequality in the Poincar\'{e} disc model of the hyperbolic plane was extended in \cite{Ne1}.  Bohr radii for  the classes of convex  univalent functions of order $\alpha$, close-to-convex functions and functions with positive real part were obtained by authors \cite{ali1}.
%In \cite{Ponnu7}, Kayumov and Ponnusamy discussed the Bohr radius for the class of analytic functions $g$ which is subordinate to a member of the class of odd univalent functions and also proved a conjecture proposed by Ali \emph{et al.} \cite{Alex}.
Powered Bohr radius for the class of analytic functions mapping the unit disc onto itself was studied in \cite{Ponnu3}. For more details, see \cite{pon1}.

%Accordingly, the Bohr radius for a class $\mathcal{G}$  consisting of analytic functions $f(z)=\sum_{n=0}^{\infty}a_nz^n$ mapping  the disc $\mathbb{D}$ into a domain $\Omega$ is  the  largest  $r^* > 0$ such that every function $f\in \mathcal{G}$ satisfies the inequality
%\[d\left(\sum_{n=0}^{\infty}|a_nz^n|, |f(0)|\right) = \sum_{n=1}^{\infty}|a_nz^n|\leq d(f(0), \partial f(\mathbb{D}))\] for all  $|z|=r \leq r^*$. In this case, the class $\mathcal{G}$ is said to  satisfy a Bohr phenomenon.

In this paper, Section \ref{sec3} provides the   Bohr inequality for the class of Janowski starlike functions. In Section \ref{sec4}, the Bohr radius problem is determined for the class of analytic functions $f$ satisfying second order differential subordination relation $f(z)+\beta z f'(z)+\gamma z^2 f''(z)\prec h(z)$, where $h$ is a Janowski  function.
Section \ref{sec5} yields the Bohr radius for the class of alpha-convex functions. In the last section,  we compute the Bohr radius for the class of typically-real function.

\section{Janowski Starlike Functions}\label{sec3}
In the present section, we obtain the Bohr's radius for the class of Janowski starlike functions. To prove our results we need the following lemmas.

Letting $p=1$ and $\a = 0$ in \cite[Theorem 3, p.\ 738]{aou},  we get following result for Janowski starlike functions.
%\cite[Theorem 3, p.\ 738]{aou}
\begin{lemma}\label{lem1}
If $f(z) = z + \sum_{n=2}^{\infty}a_n z^n \in\mathcal{ST}[A,B]$, then
\begin{equation}\label{cv1}
 | a_n| \leq \prod_{k=0}^{n-2}\frac{|(B-A)+B k|}{k+1}, (n\geq 2)
\end{equation}
and these bounds are sharp.
\end{lemma}

\begin{lemma}\label{lem2}\cite[Theorem 4, p.\ 315]{jan1}
If the function $f \in \mathcal{ST}[A,B]$, then for $|z| = r$; $(0 \le r <1)$,
\begin{equation}\label{cv2}
l_{(-A, -B)}(r)\leq\: \vert f(re^{i\theta})| \:\leq\:l_{(A, B)}(r)
\end{equation}
where $l_{(A, B)}:\mathbb{D}\rightarrow\mathbb{C}$ is given by
\begin{equation}\label{eq 1}
l_{(A, B)}(z)=\left\{ \begin{array}{ll}
                z(1+Bz)^{\frac{A-B}{B}}, & B\neq0; \\
                   ze^{Az}, & B=0.
                             \end{array}
                       \right.
\end{equation}
 The result is sharp.
\end{lemma}

\begin{theorem}\label{th4}
Let $f(z) = z + \sum_{n=2}^{\infty}a_n z^n$ be in the class $\mathcal{ST}[A,B]$. Then
$$\vert z \vert + \sum_{n=2}^{\infty}\vert a_n z^n \vert \leq d(0,\partial f(\mathbb{D})) $$ \text{for} $\vert z \vert < r^*$,
where $r^*\in(0,1]$ is the root of equations
\begin{align*}
r+\sum_{n=2}^{\infty}\prod_{k=0}^{n-2}\frac{|(B-A)+B k|}{k+1}r^n-(1-B )^{(A-B)/B}&=0,\quad \text{if}\quad B\neq0 \\
r e^{A r}-e^{-A}&=0,\quad \text{if}\quad B=0.
\end{align*}
The number $r^*$ is the Bohr radius for the class $\mathcal{ST}[A,B]$ which is best possible.
\end{theorem}
\begin{proof}
Since $ f(z) = z + \sum_{n=2}^{\infty} a_n z^n$ belong to $\mathcal{ST}[A, B]$, then
%by  Lemma \ref{lem1}, the coefficients bound for the function $f$ is given by
%\begin{equation}\label{cv1}
% | a_n| \leq \prod_{k=0}^{n-2}\frac{|(B-A)+B k|}{k+1}, (n\geq 2).
%\end{equation}
%By Lemma \ref{lem2}, the growth inequality for the function $f$ is
%\begin{equation}\label{cv2}
%l_{(-A, -B)}(r)\leq\: \vert f(re^{i\theta})| \:\leq\:l_{(A, B)}(r),\quad 0<r\leq 1,
%\end{equation}
%where the function $l_{(A, B)}$ is given by \eqref{eq 1}.
using  inequality \eqref{cv2},  the distance between boundary and the origin of the function $f$ is given by
\begin{equation}\label{cv4}
d(0,\partial f(\mathbb{D})) = \inf_{\xi\in\partial f(\D)} |f(\xi)| =l_{(-A, -B)}(1).
\end{equation}
Note that the given $r^*$ satisfies the following equations:

For $B\neq0$,
\[r+\sum_{n=2}^{\infty}\prod_{k=0}^{n-2}\frac{|(B-A)+B k|}{k+1}r^n=(1-B )^{(A-B)/B}\]

and $B=0$,
\[r e^{A r}=e^{-A}.\]
Using  \eqref{cv1}, \eqref{cv4} and the fact that $A>0$ for $B=0$, we have
 \begin{align*}\label{cv6}
 \nonumber \vert z \vert + \sum_{n=2}^{\infty}\vert a_n z^n \vert &\leq r+\sum_{n=2}^{\infty}\prod_{k=0}^{n-2}\frac{|(B-A)+B k|}{k+1}r^n\\
 &=\left\{
  \begin{array}{ll}
   r+\sum_{n=2}^{\infty}\prod_{k=0}^{n-2}\frac{|(B-A)+B k|}{k+1}r^n,& B\neq0;\\
    r e^{A r},& B=0
  \end{array}
\right.
 \\
&\leq l_{(-A, -B)}(1)\\
&=d(0,\partial f(\mathbb{D}))
\end{align*}
for $r\leq r^*$. To show the sharpness of the Bohr radius $ r^* $, consider the function $ l_{(A, B)}$ so that
for $|z|=r^*$, we have
 \begin{align*}\vert z \vert + \sum_{n=2}^{\infty}\vert a_n z^n \vert & =  r^*+\sum_{n=2}^{\infty}\prod_{k=0}^{n-2}\frac{|(B-A)+B k|}{k+1}(r^*)^n \\
 & = \begin{cases}r^*+\sum_{n=2}^{\infty}\prod_{k=0}^{n-2}\frac{|(B-A)+B k|}{k+1}(r^*)^n , B\neq0\\
 r^* e^{A r^*}, \quad B=0
\end{cases}\\
&= l_{(-A, -B)}(1)\\
&=d(0,\partial f(\mathbb{D})).
\end{align*}
 This completes the sharpness.
\end{proof}
Bhowmik and Das \cite[Theorem 3, p. 1093]{bho}  found the  Bohr radius for $S^*(\alpha)$ where $\alpha\in[0, 1/2]$.
For $0\leq\alpha<1$, $A=1-2\alpha$ and $B=-1$,  Theorem \ref{th4} gives sharp Bohr radius for the class of starlike functions of order $\alpha$.
\begin{corollary}\cite[Remark 3, p. 7]{AHU}\label{cor4} If $0\leq\alpha<1$ and $f(z)=z+\sum_{n=2}^{\infty}a_nz^n \in \mathcal{ST}(\alpha)$, then
\[\vert z \vert + \sum_{n=2}^{\infty}\vert a_n z^n \vert \leq d(0,\partial f(\mathbb{D})) \] for $\vert z \vert < r^*$, where $r^*\in(0,1]$ is the root of equation $ (1 - r)^{2(1 - \alpha)} - 2^{2(1 - \alpha)}r \geq 0.$
\end{corollary}
\begin{remark}
In particular, for $\alpha=0$, Corollary \ref{cor4} yields the sharp Bohr radius for the class of starlike functions which is $3-2 \sqrt{2}$.
\end{remark}
 Putting  $A=\beta$ and $B=-\beta$, where $0<\beta\leq1$ in  Theorem \ref{th4}, we get the  sharp Bohr radius for the class  $\mathcal{ST}^{(\beta)}$.
\begin{corollary}If $0<\beta\leq1$ and $f(z)=z+\sum_{n=2}^{\infty}a_nz^n\in \mathcal{ST}^{(\beta)}$, then
\[\vert z \vert + \sum_{n=2}^{\infty}\vert a_n z^n \vert \leq d(0,\partial f(\mathbb{D})) \] for $\vert z \vert < r^*$, where $r^*\in(0,1]$ is given as
% the root of equations
\[r^*=-\frac{-1-4 \beta -\beta ^2+(1+\beta ) \sqrt{1+\beta  (6+\beta )}}{2 \beta ^2}.\]
\end{corollary}
 If  $0<\beta\leq1$, $A=\beta$ and $B=0$, then  Theorem \ref{th4} yields sharp Bohr radius for the class  $\mathcal{ST}_{(\beta)}$.

\begin{corollary}If $f(z)=z+\sum_{n=2}^{\infty}a_nz^n\in \mathcal{ST}_{(\beta)}$, then $\vert z \vert + \sum_{n=2}^{\infty}\vert a_n z^n \vert \leq d(0,\partial f(\mathbb{D}))$ for $\vert z \vert < r^*$, where $r^*\in(0,1]$ is the root of equation $r e^{\beta r}-e^{-\beta}=0.$
\end{corollary}
Letting  $A=1$ and $B=(1-M)/M$ where $M>1/2$,  Theorem \ref{th4} provides following result for the class  $\mathcal{ST}(M)$.
\begin{corollary}If $M>1/2$ and $f(z)=z+\sum_{n=2}^{\infty}a_nz^n\in \mathcal{ST}(M)$,  then
\[\vert z \vert + \sum_{n=2}^{\infty}\vert a_n z^n \vert \leq d(0,\partial f(\mathbb{D})) \] for $\vert z \vert < r^*$, where $r^*\in(0,1]$ is the root of equation
\[ r+\sum_{n=2}^{\infty}\prod_{k=0}^{n-2}  \left(\frac{|(1-2M)+(1-M) k|}{M(k+1)}\right)r^n-\left(2-\frac{1}{M}\right)^{(1-2 M)/(-1+M)}=0.\]
\end{corollary}
%\begin{corollary}\label{cor3}Let $0\leq\alpha<1,  0<\beta\leq1$ and $ M>1/2$.   The sharp Bohr radius $r^*\in(0, 1]$ for the following classes is as follows:
%\begin{enumerate}
% \item For  $\mathcal{ST}^{(\beta)}$,
%  \[r^*=-\frac{-1-4 \beta -\beta ^2+(1+\beta ) \sqrt{1+\beta  (6+\beta )}}{2 \beta ^2}.\]
%\item For $\mathcal{ST}{(\alpha)}$, $r^*$  is  the root of the equation
%$$ (1 - r)^{2(1 - \alpha)} - 2^{2(1 - \alpha)}r \geq 0.$$
%\item For $\mathcal{ST}_{(\beta)}$,  $r^*$  is the root of the equation \[r e^{\beta r}-e^{-\beta}=0.\]
%\item For $\mathcal{ST}(M)$, $r^*$ is the root of the equation
% \[ r+\sum_{n=2}^{\infty}\prod_{k=0}^{n-2}  \left(\frac{|(1-2M)+(1-M) k|}{M(k+1)}\right)r^n-\left(2-\frac{1}{M}\right)^{\frac{1-2 M}{-1+M}}=0.\]
%\end{enumerate}
% \end{corollary}

\section{Second Order Differential Subordination Associated With Janowski  Functions} \label{sec4}

For $\beta \ge \gamma \ge 0$, we consider the class $\mathcal{R}(\beta, \gamma,h)$ which is defined by making use of subordination as
\[\mathcal{R}(\beta, \gamma,h) = \{f \in \mathcal{A}: f(z) + \beta zf'(z) +\gamma z^2 f''(z) \prec h(z), z\in \D\}\]
where $h$ is a Janowski  function. The class $\mathcal{R}(\beta, \gamma,h)$  can be seen as an extension to the class
\[\mathcal{R}(\beta, h) = \{f \in \mathcal{A}:  f'(z) +\beta z f''(z) \prec h(z), z\in \D\}.\]
Many variations of this class have been studied by various authors \cite{gao,HM,Yang}.

For two analytic functions $f(z) = \sum\limits_{n=0}^{\infty}a_nz^n$ and $g(z) = \sum\limits_{n=0}^{\infty}b_nz^n$, the Hadamard product (or convolution) is the function $f \ast g$, defined by
\[(f \ast g)(z) = \sum\limits_{n=0}^{\infty}a_nb_nz^n.\]
Consider the function $\phi_{\lambda}$, defined by
\[\phi_{\lambda}(z) = \int\limits_{0}^{1}\frac{dt}{1-zt^{\lambda}} = \sum\limits_{n=0}^{\infty} \dfrac{z^n}{1+\lambda n}.\]
For $\RE \lambda \ge 0$, the function $\phi_{\lambda}$ is convex in $\D$ \cite[Theorem 5, p.113]{RUS}.

For $\beta \ge \gamma \ge 0$, let $\nu+\mu = \beta - \gamma\,\, \text{and}\,\, \mu\nu = \gamma$
and
\begin{equation}\label{eq 2}
q(z) = \int\limits_{0}^{1} \int\limits_{0}^{1}h(zt^{\mu}s^{\nu})dt ds = (\phi_{\nu} \ast \phi_{\mu}) \ast h(z).
\end{equation}
Since $\phi_{\nu} \ast \phi_{\mu}$ is a convex function and $h \in \mathcal{ST}[A,B]$ so it follows from \cite[Theorem 5, p.167]{ma} that $ q \in \mathcal{ST}[A,B]$. The following theorem gives the sharp Bohr radius for the class $\mathcal{R}(\beta, \gamma,h)$.

\begin{theorem} \label{MT}
Let $f(z) = \sum\limits_{n=0}^{\infty}a_nz^n \in \mathcal{R}(\beta, \gamma,h)$ and $h$ be a Janowski starlike function. Then
\[\sum\limits_{n=1}^{\infty}|a_n z^n |\le d(h(0),\partial h(\mathbb{D}))\]
for all $|z| \le r^*$,  where $r^* \in(0,1]$ is root of the equations
\[\dfrac{r}{1+(\nu+\mu) + \nu \mu } +\sum\limits_{n=2}^{\infty}\dfrac{\prod\limits_{k=0}^{n-2}\dfrac{|(B-A)+ Bk|}{k+1}}{1+(\nu+\mu)n + \nu \mu n^2}r^n = l_{(-A,-B)}(1),\]
where
\[l_{(-A,-B)}(1)= \left\{ \begin{array}{ll}
                 (1-B)^{\frac{A-B}{B}}, & \mbox{$B \ne 0$}; \\
                    e^{-A} , & \mbox{$B=0$}
                              \end{array}
                              \right. \]
The result is sharp.
\end{theorem}

\begin{proof}
Let $F(z) = f(z) +\beta zf'(z) + \gamma z^2 f''(z) \prec h(z)$. It is noted that
\[F(z) = \sum\limits_{n=0}^{\infty}[1+\beta n+ \gamma n(n-1)]a_n z^n \prec h(z), \quad a_1=1.\]
Consider
\[\frac{1}{h'(0)}\sum\limits_{n=1}^{\infty}[1+\beta n+ \gamma n(n-1)]a_nz^n = \frac{F(z)-F(0)}{h'(0)} \prec \frac{h(z)-h(0)}{h'(0)}= H(z).\]
Since $h \in \mathcal{ST}[A,B]$, it follows that $H \in \mathcal{ST}[A,B]$. Thus, Lemma \ref{lem1} gives
\[\left|\dfrac{1+\beta n+ \gamma n(n-1)}{h'(0)}\right||a_n| \le \prod\limits_{k=0}^{n-2}\dfrac{|(B-A)+Bk|}{k+1}\]
for each $n \ge 2$.
In view of the above inequality, we have
\begin{equation}\label{an}
\sum\limits_{n=1}^{\infty}|a_n|r^n \le \dfrac{|h'(0)|}{1+(\nu +\mu)+ \mu \nu }r+ \sum\limits_{n=2}^{\infty}\dfrac{|h'(0)|\prod\limits_{k=0}^{n-2}\dfrac{|(B-A)+Bk|}{k+1}}{1+(\nu +\mu)n+ \mu \nu n^2}r^n.
\end{equation}
Since  $H \in \mathcal{ST}[A,B]$, using \eqref{eq 1} and \eqref{cv2} the following inequality holds
\begin{equation} \label{H(z)}
l_{(-A,-B)}(r)  \le |H(r \E)| \le l_{(A,B)}(r),\,\, 0 < r \le 1.
\end{equation}
So that
\[d(0, \partial H(\D)) \ge l_{(-A,-B)}(1)  \]
which gives
\begin{equation}\label{h(z)}
d(h(0), \partial h(\D)) = \inf_{\xi\in\partial h(\D)} |h(\xi) -h(0)| \ge |h'(0)|l_{(-A,-B)}(1).
\end{equation}

Using(\ref{an}) and (\ref{h(z)}) we obtain
\begin{equation*}
\sum\limits_{n=1}^{\infty}|a_n|r^n \le \dfrac{d(h(0), \partial h(\D))}{l_{(-A,-B)}(1)} \left(\dfrac{1}{1+(\nu +\mu)+ \mu \nu}r+ \sum\limits_{n=2}^{\infty}\dfrac{\prod\limits_{k=0}^{n-2}\dfrac{|
(B-A)+Bk|}{k+1}}{1+(\nu +\mu)n+ \mu \nu n^2}r^n\right).\\
 \end{equation*}
Thus, the  Bohr radius $r^*$ is the smallest postive root of the equation
 \[\dfrac{r}{1+(\nu+\mu) + \nu \mu} +\sum\limits_{n=2}^{\infty}\dfrac{\prod\limits_{k=0}^{n-2}\dfrac{|(B-A)+ Bk|}{k+1}}{1+(\nu+\mu)n + \nu \mu n^2}r^n = l_{(-A,-B)}(1). \]
For sharpness, consider the function $f(z) = q(z) = (\phi_{\nu} \ast \phi_{\mu}) \ast h(z)$ as defined in \eqref{eq 2},
where
\[h(z) = l_{(A,B)}(z) = \left\{ \begin{array}{ll}
                z(1+Bz)^{\frac{A-B}{B}}, & \mbox{$B \ne 0$}; \\
                   ze^{Az} , & \mbox{$B=0.$}
                             \end{array}
                       \right. \]
Also $f(z) \in \mathcal{R}(\beta,\gamma,h)$.
This gives
\begin{align*}
f(z) = \dfrac{z}{1+(\nu+\mu) + \nu \mu } +\sum\limits_{n=2}^{\infty}\dfrac{\prod\limits_{k=0}^{n-2}\dfrac{|(B-A)+ Bk|}{k+1}}{1+(\nu+\mu)n + \nu \mu n^2}z^n.
\end{align*}

For $|z| = r^*$,
\begin{align*}
\sum\limits_{n=1}^{\infty}|a_nz^n| &= |h'(0)|\left(\dfrac{r^*}{1+(\nu+\mu) + \nu \mu} +\sum\limits_{n=2}^{\infty}\dfrac{\prod\limits_{k=0}^{n-2}\dfrac{|(B-A)+ Bk|}{k+1}}{1+(\nu+\mu)n + \nu \mu n^2}(r^*)^n\right)\\
& = |h'(0)| l_{(-A,-B)}(1)\\
& = d(h(0), \partial h(\D)).
\end{align*}
Thus the result is sharp.
\end{proof}

For $0 \le \alpha <1, A = 1-2 \alpha$ and $B=-1$, Theorem \ref{MT} reduces to the following result.
\begin{corollary} \cite [Theorem 3.3, p.7] {NKJ}
Let $f(z) = \sum\limits_{n=0}^{\infty}a_nz^n \in \mathcal{R}(\beta, \gamma,h)$ and $h$ be starlike of order $\alpha$. Then
\[\sum\limits_{n=1}^{\infty}|a_n z^n| \le d(h(0),\partial h(\mathbb{D}))\]
for all $|z| \le r^*$,  where $r^* \in(0,1]$ is the smallest positive root of the equation
\[\dfrac{r}{1+(\nu+\mu) + \nu \mu} +\sum\limits_{n=2}^{\infty}\dfrac{\frac{1}{(n-1)!}\prod\limits_{k=2}^{n}(k-2 \alpha)}{1+(\nu+\mu)n + \nu \mu n^2}r^n = \dfrac{1}{2^{2(1-\alpha)}}.\]

The result is sharp.
\end{corollary}

For $ A = 1$ and $B=-1$, Theorem \ref{MT} yields  the following result.
\begin{corollary}\cite [Theorem 3.3, p.131] {ABU}
Let $f(z) = \sum\limits_{n=0}^{\infty} a_n z^n \in \mathcal{R}(\beta, \gamma,h)$ and $h$ be starlike. Then
\[\sum\limits_{n=1}^{\infty}|a_n z^n| \le d(h(0),\partial h(\mathbb{D}))\]
for all $|z| \le r^*$,  where $r^* \in(0,1]$ is the smallest positive root of the equation
\[\sum\limits_{n=1}^{\infty}\dfrac{n}{1+(\nu+\mu)n + \nu \mu n^2}r^n = \dfrac{1}{4}.\]

The result is sharp.
\end{corollary}

\section{Alpha-Convex Functions} \label{sec5}

In 1969, Mocanu \cite{MOCANU} introduced the class of $\alpha$-convex functions. For $\alpha\in\mathbb{R}$, a normalized analytic function $f(z) = z+ \sum\limits_{n=2}^{\infty}a_n z^n$ is said to be $\alpha$-convex in $\D$(or $\a$-convex) if the following conditions hold
\[\dfrac{f(z)}{z}f'(z) \ne 0,\]
and
\[\RE\left[\a \left( 1+\dfrac{zf''(z)}{f'(z)}\right) +(1-\a) \dfrac{zf'(z)}{f(z)} \right] \ge 0, \]
for all $z \in \D$. The set of all such functions is denoted by $\a-\mathcal{CV}$. For $\a = 0$ , $\a-\mathcal{CV}$ is the class of starlike functions and for $\a =1$, $\a-\mathcal{CV}$ is the class of convex functions.

In this section, we shall obtain the Bohr radius for the class of $\a$-convex functions. In order to obtain the Bohr radius for the class $\a-\mathcal{CV}$, we need the following lemmas.

\begin{lemma} \cite [Theorem 7, p. 146] {GOODMAN}\label{distCV}
If $\a >0$, and $f(z)$ is $\a$-convex, then
\[k(-r,\a) \le |f(z)| \le k(r,\a), \,\,z= r \E,\]
where
\begin{equation}\label{eq 3}
k(z,\a) = \left[\frac{1}{\a}\int\limits_{0}^{z}\frac{\xi^{\frac{1}{\a}-1}d \xi}{(1-\xi)^{\frac{2}{\a}}}\right]^\a.
 \end{equation}
 The inequalites are sharp for each $\a >0$ and each $r \in(0,1)$.
\end{lemma}

\begin{lemma}\cite [Theorem 2, p. 208] {KUL}
Let $f(z) = \sum_{n=0}^\infty a_nz^n \in \a-\mathcal{CV}$ and let $S(n)$ be the set of all $n$-tuples $(x_1,x_2, \cdots ,x_n)$ of non-negative integers for which $\sum\limits_{i=1}^{n}i x_i =n$ and for each $n$-tuple define $q$ by $\sum\limits_{i=1}^{n}x_i = q$. If \[\vU(\a,q) = \a(\a-1)(\a-2) \cdots (\a-q)\]
with
$\vU(\a,0) = \a$, then for $n =1,2, \cdots$
\begin{equation} \label{coeffCV}
|a_{n+1}| \le \sum\dfrac{\vU(\a,q-1)c_1^{x_1}c_2^{x_2} \cdots c_n^{x_n}}{x_1!x_2! \cdots x_n!},
\end{equation}
where summation is taken over all $n$-tuples in $S(n)$ and
\[c_n = \frac{1}{n!\a^n(1+n \a) } \prod\limits_{k=0}^{n-1}(2+k \a).\]
The result is sharp.
\end{lemma}

\begin{theorem}
Let $\a>0$ and $f(z) = z+\sum\limits_{n=2}^{\infty}a_n z^n \in \a-\mathcal{CV}$. Then
\[|z|+\sum\limits_{n=2}^{\infty}|a_n z^n| \le d(0,\partial f(\D))\]
for $|z| \le r^*$, where $r^* \in (0,1]$ is the positive root of the equation
\[r+ \sum\limits_{n=2}^{\infty}\left(\sum\dfrac{\vU(\a,q-1)c_1^{x_1}c_2^{x_2} \cdots c_{n-1}^{x_{n-1}}}{x_1!x_2! \cdots x_{n-1}!}\right)r^n = k(-1,\a)\] and where summation is as in \eqref{coeffCV}.
The result is sharp.
\end{theorem}

\begin{proof} Let $f\in \a-\mathcal{CV}$.
By Lemma \ref{distCV}, the growth inequality for the function $f $ is given by
\[k(-r,\a) \le |f(r \E)| \le k(r,\a), \] where $k$ is the function  as defined  in \eqref{eq 3}.
This immediately shows that
\begin{equation} \label{eqdistCV}
d(0,\partial f(\D)) \geq k(-1,\a).
\end{equation}
It is given that the Bohr radius $r^*$ is the root of the equation
\[r+ \sum\limits_{n=2}^{\infty}\left(\sum\dfrac{\vU(\a,q-1)c_1^{x_1}c_2^{x_2} \cdots c_{n-1}^{x_{n-1}}}{x_1!x_2! \cdots x_{n-1}!}\right)r^n = k(-1,\a).\]
For $0\leq r\leq r^*$, it is readily seen that $k(r,\a)\leq k(-1,\a)$.
Using (\ref{coeffCV}) and (\ref{eqdistCV}),we have

\begin{align*}
|z|+\sum\limits_{n=2}^{\infty}|a_n z^n|& \le r+ \sum\limits_{n=2}^{\infty}\left(\sum\dfrac{\vU(\a,q-1)c_1^{x_1}c_2^{x_2} \cdots c_{n-1}^{x_{n-1}}}{x_1!x_2! \cdots x_{n-1}!}\right)r^n
 \le k(-1,\a)
 \le d(0,\partial f(\D))
\end{align*}
for $|z| =r \le r^*$.
In order to prove the sharpness, consider the function
\[f(z)=k(z,\a)= \left[\dfrac{1}{\a}\int\limits_{0}^{z}\dfrac{\xi^{\frac{1}{\a}-1}d \xi}{(1-\xi)^{\frac{2}{\a}}}\right]^\a.\]
For $|z| = r^*$, we obtain
\begin{align*}
|z|+\sum\limits_{n=2}^{\infty}|a_n z^n|=r^*+ \sum\limits_{n=2}^{\infty}\left(\sum\dfrac{\vU(\a,q-1)c_1^{x_1}c_2^{x_2} \cdots c_{n-1}^{x_{n-1}}}{x_1!x_2! \cdots x_{n-1}!}\right)(r^*)^n&=k(-1,\a)\\
&=d(0,\partial f(\D)).
\end{align*}
Thus the result is sharp.
\end{proof}

\section{Typically Real Functions} \label{sec6}

The class of typically real functions was introduced by Rogosinski \cite{ROG}. An analytic function $f(z) = z+\sum\limits_{n=2}^{\infty} a_nz^n$ which satisfies the condition $\text{sign}(\IM f(z)) = \text{sign}(\IM z)$
for non real $z \in \D$  is said to be typically real in $\D$. The class of such functions is denoted by $\mathcal{TR}$.
In the present section, we obtain the Bohr radius for the class of typically real functions.

\begin{lemma}\cite [Theorem 3, p. 185] {GOODMAN}\label{coeffTR}
If $f(z)$ is in $\mathcal{TR}$ and $z = r\E \in \D$, then the coefficients satisfy the inequality
\begin{equation} \label{eqnTR}
m_n \le a_n \le n
\end{equation}
where $m_n = \min(\sin n\theta/\sin \theta)$ for each $n$. The inequality is sharp for each $n$.
\end{lemma}

\begin{lemma}\cite[Thm 1, p. 136] {REMI}\label{lem3}
Let the function $f$ be in $\mathcal{TR}$ and $z = r \E \in \D$. Then
\begin{equation} \label{eqnTR2}
|f(z)| \ge  \left\{ \begin{array}{ll}
                \left|\dfrac{z}{(1+z)^2}\right|, & \mbox{if $\RE\left(\dfrac{1+z^2}{z}\right) \ge 2$}; \\
                   \dfrac{r(1-r^2)|\sin \theta|}{|1-z^2|^2}, & \mbox{if $\RE\left|\dfrac{1+z^2}{z}\right| \le 2$};\\
                   \left|\dfrac{z}{(1-z)^2}\right|, & \mbox{if $\RE\left(\dfrac{1+z^2}{z}\right) \le -2.$}
                             \end{array}
                       \right.
                       \end{equation}
 The result is sharp.
\end{lemma}

The next theorem gives the sharp Bohr radius for the class $\mathcal{TR}$.

\begin{theorem}
Let $f(z) = z+ \sum\limits_{n=2}^{\infty}a_nz^n$ be in the class $\mathcal{TR}$ and $z = r \E \in \D$. Then
\[|z|+\sum\limits_{n=2}^{\infty}|a_n z^n| \le d(0,\partial f(\D))\]
for $|z| \le 3-2\sqrt2\approx0.171573$. The result is sharp.
\end{theorem}

\begin{proof}
Let $f(z) = z +\sum\limits_{n=2}^{\infty}a_nz^n \in \mathcal{TR}$.
By Lemma \ref{lem3} and the proof of Theorem 2 (see \cite{bra}),  it follows that the distance between the origin and the
boundary of $f(\D)$ satisfies the inequality
\begin{equation}\label{eq 4}
|f(z)| \ge \frac{1}{4}, z \in \D.
\end{equation}
Using the inequality \eqref{eq 4}, we have
\begin{equation} \label{eqdistTR}
d(0,\partial f(\D)) = \inf_{\xi\in\partial f(\D)} |f(\xi)| \ge \frac{1}{4}.
\end{equation}
In view of inequalities (\ref{eqnTR}) and (\ref{eqdistTR}), we have
\begin{align*}
|z|+\sum\limits_{n=2}^{\infty}|a_n z^n| & \le r+ \sum\limits_{n=2}^{\infty}n r^n =\frac{r}{(1-r)^2}\le \frac{1}{4}\le d(0,\partial f(\D))
\end{align*}
if $r<3-2\sqrt2$.
In order to prove the sharpness, for $-\pi<t\leq \pi$, consider the function $l_t:\D\rightarrow\C$ defined by
\[l_t(z) = \frac{z}{(1-2z\cos t+z^2)}.\]
 For $|z| = r^*$ and $t=0$, we obtain
\begin{align*}
|z|+\sum\limits_{n=2}^{\infty}|a_n z^n| &= r^*+ \sum\limits_{n=2}^{\infty}n(r^*)^n =\frac{r^*}{(1-r^*)^2}
=\frac{1}{4} = d(0,\partial f(\D)).
\end{align*}
Thus the result is sharp.
\end{proof}


\begin{thebibliography}{99}
\bibitem{pon1}Y. Abu Muhanna, R. M. Ali\ and\ S. Ponnusamy, On the Bohr inequality, in  "Progress in approximation theory and applicable complex analysis",  Springer Optim. Appl.  {\bf 117} (2016),  265--295.
\bibitem{ABU}Y. Abu Muhanna, R. M. Ali, Z. C. Ng\ and\  M. F. Siti, Bohr radius for subordinating families of analytic functions and bounded harmonic mappings, J. Math. Anal. Appl. {\bf 420} (2014), no.~1, 124--136.

\bibitem{AHU}O. P. Ahuja, S. Anand\ and\ N. K. Jain, Bohr radius problems for some classes of analytic functions using quantum calculus approach, Mathematics {\bf 8} (2020), 623.

\bibitem{Alex} R. M. Ali, R. W. Barnard\ and\ A. Yu. Solynin, A note on Bohr's phenomenon for power series, J. Math. Anal. Appl. {\bf 449} (2017), no.~1, 154--167.
\bibitem{Ne1}R. M. Ali\ and\ Z. C. Ng, The Bohr inequality in the hyperbolic plane, Complex Var. Elliptic Equ. {\bf 63} (2018), no.~11, 1539--1557.

\bibitem{ali1}R. M. Ali, N. K. Jain\ and\ V. Ravichandran, Bohr radius for classes of analytic functions, Results Math. {\bf 74} (2019), no.~4, Art. 179, 13pp.

\bibitem{Ponnu1}S. A. Alkhaleefah, I. R. Kayumov\ and\ S. Ponnusamy, On the Bohr inequality with a fixed zero coefficient, Proc. Amer. Math. Soc. {\bf 147} (2019), no.~12, 5263--5274.
%\bibitem{and}G. E. Andrews, R. Askey\ and\ R. Roy, {\it Special functions}, Encyclopedia of Mathematics and its Applications, 71, Cambridge University Press, Cambridge, 1999.
\bibitem{aou}M. K. Aouf, On a class of $p$-valent starlike functions of order $\alpha$, Internat. J. Math. Math. Sci. {\bf 10} (1987), no.~4, 733--744.
\bibitem{bho} B. Bhowmik\ and\ N. Das, Bohr phenomenon for subordinating families of certain univalent functions, J. Math. Anal. Appl. {\bf 462} (2018), no.~2, 1087--1098.
\bibitem{B1}H. Bohr,  A theorem concerning power series,  Proc. London Math. Soc. (2) {\bf 13} (1914),  1--5.
\bibitem{bra}D. A. Brannan\ and\ W. E. Kirwan, A covering theorem for typically real functions, Glasgow Math. J. {\bf 10} (1969), 153--155.
\bibitem{D1}P. G. Dixon,  Banach algebras satisfying the non-unital von Neumann inequality,  Bull. London Math. Soc. {\bf 27} (1995),  no.~4,  359--362.
%\bibitem{dur}P. L. Duren, {\it Univalent functions}, Grundlehren der Mathematischen Wissenschaften, 259, Springer-Verlag, New York, 1983.
%\bibitem{evd} S. Evdoridis, S. Ponnusamy\ and\ A. Rasila, Improved Bohr's inequality for locally univalent harmonic mappings, Indag. Math. (N.S.) {\bf 30} (2019), no.~1, 201--213.
\bibitem{gao} C. Y. Gao\ and\ S. Q. Zhou, Certain subclass of starlike functions, Appl. Math. Comput. {\bf 187} (2007), no.~1, 176--182.

\bibitem{GOODMAN} A. W. Goodman, {\it Univalent functions. Vol. I}, Mariner Publishing Co., Inc., Tampa, FL, 1983.

\bibitem{NKJ} N. K. Jain and S. Yadav, Bohr radius for certain analytic functions, in "Mathematical Analysis I: Approximation Theory" {\bf 306} (2018), 211--221,  Springer Proceedings in Mathematics and Statistics, Springer, Singapore.
\bibitem{jan2}W. Janowski, Extremal problems for a family of functions with positive real part and for some related families, Ann. Polon. Math. {\bf 23} (1970/1971), 159--177.
\bibitem{jan1} W. Janowski, Some extremal problems for certain families of analytic functions. I, Ann. Polon. Math. {\bf 28} (1973), 297--326.
\bibitem{Ponnu3} I. R. Kayumov\ and\ S. Ponnusamy, On a powered Bohr inequality, Ann. Acad. Sci. Fenn. Math. {\bf 44} (2019), no.~1, 301--310.
\bibitem{Ponnu5}I. R. Kayumov and S. Ponnusamy, Improved Version of Bohr's Inequality,   C. R. Math. Acad. Sci. Paris  {\bf 356}(2018), no.~3, 272--277.
%\bibitem{Ponnu4}I. R. Kayumov  and S. Ponnusamy, Bohr's inequalities for the analytic functions with lacunary series and harmonic functions,   J. Math. Anal. and Appl. {\bf 2018}, 465, 857--871.
%%\bibitem{Ponnu7} I. R. Kayumov and  S. Ponnusamy,  Bohr inequality for odd analytic functions, Comput. Methods Funct. Theory {\bf2017}, 17, 679--688.
%\bibitem{Ponnu6}I. R. Kayumov, S. Ponnusamy and N. Shakirov,  Bohr radius for locally univalent harmonic mappings, Math. Nachr. {\bf 2018}, 291,  1757--1768.
\bibitem{KUL}P. K. Kulshrestha, Coefficient problem for alpha-convex univalent functions, Arch. Rational Mech. Anal. {\bf 54} (1974), 205--211.
\bibitem{ma}W. C. Ma\ and\ D. Minda, A unified treatment of some special classes of univalent functions, in {\it Proceedings of the Conference on Complex Analysis (Tianjin, 1992)}, 157--169, Conf. Proc. Lecture Notes Anal., I, Int. Press, Cambridge.
\bibitem{mac}T. H. MacGregor, The radius of univalence of certain analytic functions, Proc. Amer. Math. Soc. {\bf 14} (1963), 514--520.
\bibitem{MOCANU}P. T. Mocanu, Une propri\'{e}t\'{e} de convexit\'{e} g\'{e}n\'{e}ralis\'{e}e dans la th\'{e}orie de la repr\'{e}sentation conforme, Mathematica (Cluj) {\bf 34} (1969), no.~11, 127--133.
\bibitem{P2}V. I. Paulsen,  G. Popescu\ and\ D. Singh,  On Bohr's inequality,  Proc. London Math. Soc. (3) {\bf 85} (2002),  no.~2,  493--512.
\bibitem{P3}V. I. Paulsen\ and\ D. Singh,  Bohr's inequality for uniform algebras,  Proc. Amer. Math. Soc. {\bf 132} (2004),  no.~12,  3577--3579.
\bibitem{P4}V. I. Paulsen\ and\ D. Singh,  Extensions of Bohr's inequality,  Bull. London Math. Soc. {\bf 38} (2006),  no.~6,  991--999.
\bibitem{REMI} M. P. Remizova, Extremal problems in the class of typically real functions, Izv. Vys\v{s}. U\v{c}ebn. Zaved. Matematika {\bf 32} (1963), no.~1 , 135--144.\bibitem{rob}M. I. S. Robertson, On the theory of univalent functions, Ann. of Math. (2) {\bf 37} (1936), no.~2, 374--408.
\bibitem{ROG}W. Rogosinski, \"{U}ber positive harmonische Entwicklungen und typisch-reelle Potenzreihen, Math. Z. {\bf 35} (1932), no.~1, 93--121.
\bibitem{RUS}S. Ruscheweyh, New criteria for univalent functions, Proc. Amer. Math. Soc. {\bf 49} (1975), 109--115.
\bibitem{HM} H. M. Srivastava, D. R\u{a}ducanu\ and\ P. Zaprawa, A certain subclass of analytic functions defined by means of differential subordination, Filomat {\bf 30} (2016), no.~14, 3743--3757.
\bibitem{Yang} D. G. Yang\ and\ J. L. Liu, A class of analytic functions with missing coefficients, Abstr. Appl. Anal. {\bf 2011}, Art. ID 456729, 16 pp.

\end{thebibliography}
\end{document}